\documentclass[12pt]{amsart}
\usepackage{amsmath,amssymb}
\newcommand{\p}{\varphi}
\newcommand{\cB}{{\mathcal B}}
\newcommand{\cU}{{\mathcal U}}
\newcommand{\pt}{\mathbin{\hat\otimes}}
\DeclareMathOperator{\id}{id}

\newtheorem*{thm}{Theorem}
\newtheorem*{prop}{Proposition}
\newtheorem*{lem}{Lemma}
\newtheorem*{cor}{Corollary}

\title{A remark on contractible Banach algebras}
\author{Narutaka Ozawa}
\address{RIMS, Kyoto University, \mbox{606-8502} Japan}
\email{narutaka@kurims.kyoto-u.ac.jp}

\date{27 October 2011}

\begin{document}
\begin{abstract}
It is a longstanding problem whether every contractible Banach algebra
is necessarily finite-dimensional. In this note, we confirm this for
Banach algebras acting on Banach spaces with the uniform approximation property.
This generalizes a result of Paulsen and Smith
who proved the same for Banach algebras acting on Hilbert spaces.
\end{abstract}

\maketitle

A unital Banach algebra $A$ is said to be \emph{contractible} or \emph{super-amenable}
if $H^1(A,V)=0$ for every Banach $A$-bimodule $V$; or equivalently, if $A$ has
a diagonal. A \emph{diagonal} is an element $\sum_{i=1}^\infty c_i\otimes d_i$
in the projective tensor product $A \pt A$ such that
\[
\sum_{i=1}^\infty \|c_i\|\|d_i\|<+\infty,\
\sum_{i=1}^\infty ac_i\otimes d_i=\sum_{i=1}^\infty c_i\otimes d_ia
\mbox{ for every $a\in A$, and }\sum_{i=1}^\infty c_i d_i=1.
\]
It is a longstanding problem whether every contractible Banach algebra
is finite-dimensional (and hence is a finite direct sum of full matrix algebras).
See Section~4.1 in~\cite{runde}.

For a Banach algebra $A$, we define the ideal of
``compact right multipliers'' as
\[
K_r(A)=\{ x\in A : \mbox{the operator $A\ni a\mapsto ax\in A$ is compact}\}.
\]
Observe that $K_r(A)$ is indeed a closed two-sided ideal of $A$.
Likewise, one considers the ideal $K_l(A)$ of compact left multipliers,
which need not be the same as $K_r(A)$.
A unital Banach algebra $A$ is finite-dimensional if and only if $K_r(A)=A$.
For example, if $A=c$ is the Banach algebra of convergent sequences,
then $K_r(A)=c_0$ and $K_r(A)$ has codimension $1$ in $A$.
This phenomenon does not occur in case $A$ is contractible.
Indeed, this follows from the following standard fact, which we put a proof for
the reader's convenience.
\begin{lem}
Let $A$ be a contractible Banach algebra, $V$ a left $A$-module\footnote{
All modules in this note are Banach modules.} and $W\subset V$
an $A$-submodule. If $W$ is complemented\footnote{
Complemented $=$ a direct summand $=$ the range of a bounded linear projection.}
in $V$ as a Banach space, then it is complemented as an $A$-module.
If a closed left ideal $L$ of $A$ is complemented in $A$
(e.g., if $A/L$ is finite-dimensional), then
there is an idempotent $e$ such that $L=Ae$.
If $L$ is moreover a two-sided ideal, then the idempotent $e$ is central.
\end{lem}
\begin{proof}
Out of a bounded linear projection $P$ from $V$ onto $W$, one obtains
a left $A$-module projection $Q$, by $Q\colon x\mapsto\sum c_iP(d_ix)$.
Now, let $V=A$ and $W=L$. Then, $e=Q(1)$ is an idempotent such that $L=Ae$.
When $L$ is a two-sided ideal, one has $ex=exe$ for every $x\in A$ and
$e=\sum ec_id_i=\sum ec_ied_i=\sum c_ied_ie=\sum c_ied_i$, which is central.
\end{proof}

Thus, every finite-dimensional representation of a contractible Banach
algebra $A$ factors through $A\to Ae$ where $e$ is a central idempotent
in $K_r(A)$. In particular, if an infinite-dimensional
contractible Banach algebra $A$ exists, then $A/K_r(A)$ is a contractible
Banach algebra having no finite-dimensional representations. Further taking
a quotient by a maximal ideal, one obtains an infinite-dimensional quotient
of $A$ which is simple.

Recall that a Banach space $V$ is said to have the
\emph{approximation property}, or AP in short, if there is a net $(\p_k)$ of
finite-rank operators on $V$ which converges to $\id_V$ uniformly on compact
subsets. See~\cite{handbook}.

\begin{prop}
Let $A$ be a contractible Banach algebra and $V$ a left $A$-module
with the AP. Then, $K_r(A)V$ is dense in $V$. In particular,
$A\ni a\mapsto av\in V$ is compact for every $v\in V$.
\end{prop}
\begin{proof}
Let $v\in V$. Take a net $(\p_k)$ of finite-rank operators on $V$
as above.
We may assume that $\sum\|c_i\|<+\infty$ and $\|d_i\|\to0$.
Thus $v_k:=\sum_i c_i\p_k(d_iv)\to\sum_i c_id_iv=v$ as $k$ tends to infinity.
Fix $k$ and expand $\p_k$ as $\sum_j \xi_j\otimes w_j\in V^*\otimes_{\mathrm{alg}} V$.
Then, $v_k=\sum_j \bigl(\sum_i c_i\xi_j(d_iv)\bigr)w_j$.
We observe that $x:=\sum_i c_i\eta(d_i)\in K_r(A)$ for every $\eta\in A^*$.
Indeed, this follows from the fact
$ax=\sum_i c_i\eta(d_ia)=\lim_n \sum_{i=1}^n c_i\eta(d_ia)$,
where the convergence is uniform for $a$ in the unit ball of $A$.
Therefore, $\sum_i c_i\xi_j(d_iv)\in K_r(A)$ for every $j$, and $v_k\in K_r(A)V$.
\end{proof}

\begin{cor}[\cite{taylor}]
Every contractible Banach algebra with the AP is finite-dimensional.
\end{cor}
\begin{cor}
Every contractible Banach algebra $A$ with $K_r(A)=0$
(e.g., $A$ is simple and infinite-dimensional) does not act
on a non-zero Banach space with the AP.
\end{cor}

A classical theorem of Bernard and Cole states that if $A\subset\cB(H)$ is a
Banach subalgebra of bounded linear operators on a Hilbert space $H$ and
$I\subset A$ is a closed two-sided ideal, then the quotient Banach algebra $A/I$
is again (isometrically isomorphic to) a Banach subalgebra of $\cB(\tilde{H})$
for some Hilbert space $\tilde{H}$. Since Hilbert spaces have the AP, this fact
together with the above results recovers Paulsen and Smith's theorem
(\cite{paulsen-smith}) that every contractible Banach algebra acting on
a Hilbert space is finite-dimensional. On the other hand, it is not known
whether Bernard and Cole's theorem extends to, for example, $L_p$-spaces
(see~\cite{lemerdy}). Thus, it requires an additional work to extend
Paulsen and Smith's theorem from a Hilbert space to an $L_p$-space.
Recall that a Banach space $V$ has the
\emph{uniform Grothendieck approximation property}, or UGAP in short,
if every ultrapower of $V$ has the AP (see~\cite{handbook}).
All $L_p$-spaces $(1\le p\le\infty)$
have the UGAP. More generally, UGAP-valued $L_p$-spaces still have the UGAP.

\begin{thm}
Let $A$ be a contractible Banach algebra and $V$ a left $A$-module
with the UGAP. Then, the image of $A$ in $\cB(V)$ is finite-dimensional.
\end{thm}
\begin{proof}
Assume that the conclusion does not hold.
We may assume that $A$ is an infinite-dimensional Banach subalgebra of $\cB(V)$ and
choose a weakly null net $(a_n)_{n\in I}$ in $A$ such that $\|a_n\|=1$ for all $n$.
Let $\cU$ be a cofinal ultrafilter on the index set $I$ and consider
the ultrapower $V^\cU$ of $V$. Then, $A$ acts on $V^\cU$ diagonally.
It induces a \emph{right} $A$ action on the dual Banach space $(V^\cU)^*$,
given by $(\xi a)(v)=\xi(av)$ for $\xi\in (V^\cU)^*$, $a\in A$ and $v\in V^\cU$.
Since $V$ has the UGAP, $(V^\cU)^*$ has the AP by Heinrich's theorem.
Indeed, this follows from the fact that $(V^\cU)^{**}$ has the UGAP
(see Theorem 9.1 and Proposition 6.7 in \cite{heinrich}).
Choose a net $(v_n)$ of unit vectors in $V$ such that $\|a_n v_n\|\to1$
and let $w\in V^\cU$ be the unit vector represented by $(a_nv_n)$.
Since for every compact \emph{left} multiplier $x\in K_l(A)$
the net $(xa_n)$ is norm null, one has $K_l(A)w=0$ and
$(V^\cU)^*K_l(A)\subset \{w\}^\perp$.
However, this is in contradiction with the above proposition applied to
the right $A$-module $(V^\cU)^*$, which reads that $(V^\cU)^*K_l(A)$
is dense in $(V^\cU)^*$.
\end{proof}

\subsection*{Acknowledgment}
This research was carried out while the author was visiting at
the Institut Henri Poincar\'e (IHP) for the Program 
``von Neumann algebras and ergodic theory of group actions,''
or having lunch with Nicolas Monod at a salad bar.
The author gratefully acknowledges the kind hospitality and
stimulating environment provided by IHP and the program organizers.
He would like to thank N. Monod for encouraging him to work on the contractible 
Banach algebra problem and G.~Pisier for drawing
his attention to the paper~\cite{paulsen-smith}.
Research partially supported by IHP, JSPS and Sumitomo Foundation.

\end{document}